\documentclass[12pt]{amsart}

\usepackage{amsthm, amssymb, booktabs, fullpage}

\newtheorem{theorem}{Theorem}
\newtheorem{lemma}{Lemma}[section]

\theoremstyle{remark}

\numberwithin{equation}{section}

\newcommand{\gam}{\gamma}
\newcommand{\eps}{\varepsilon}
\newcommand{\tet}{\theta}

\begin{document}

\title[The Waring--Goldbach problem for large powers]{On the Waring--Goldbach 
problem\\ for eighth and higher powers}
\author[A.V. Kumchev]{Angel V. Kumchev}
\address{Department of Mathematics\\ Towson University\\ Towson, MD 21252\\ USA}
\email{akumchev@towson.edu}
\author[T.D. Wooley]{Trevor D. Wooley}
\address{School of Mathematics\\ University of Bristol\\ University Walk\\ 
Bristol~BS8~1TW\\ UK}
\email{matdw@bristol.ac.uk}
\subjclass[2000]{11P32, 11L20, 11P05, 11P55.}

\begin{abstract} Recent progress on Vinogradov's mean value theorem has resulted in 
improved estimates for exponential sums of Weyl type. We apply these new estimates to 
obtain sharper bounds for the function $H(k)$ in the Waring--Goldbach problem. We 
obtain new results for all exponents $k\ge 8$, and in particular establish that 
$H(k)\le (4k-2)\log k+k-7$ when $k$ is large, giving the first improvement on the 
classical result of Hua from the 1940s.
\end{abstract}
\maketitle

\section{Introduction} A formal application of the Hardy--Littlewood method suggests that 
whenever $s$ and $k$ are natural numbers with $s\ge k+1$, then all large integers $n$ 
satisfying appropriate local conditions should be represented as the sum of $s$ $k$th 
powers of prime numbers. With this expectation in mind, consider a natural number $k$ 
and prime number $p$, and define $\tet=\tet(k,p)$ to be the integer with $p^\tet|k$ 
but $p^{\tet+1}\nmid k$, and $\gam=\gam(k,p)$ by
$$\gam(k,p)=\begin{cases} \tet+2,&\text{when $p=2$ and $\tet>0$,}\\
\tet+1,&\text{otherwise.}
\end{cases}$$
We then put
$$K(k)=\prod_{(p-1)|k}p^\gam,$$
and denote by $H(k)$ the least integer $s$ such that every sufficiently large positive 
integer congruent to $s$ modulo $K(k)$ may be written in the shape
\begin{equation}\label{1.1}
p_1^k+p_2^k+\ldots +p_s^k=n,
\end{equation}
with $p_1,\dots,p_s$ prime numbers. We note that the local conditions introduced in the 
definition of $H(k)$ are designed to exclude the degenerate situations in which one or 
more variables might otherwise be forced to be prime divisors of $K(k)$. In such 
circumstances, the representation problem at hand reduces to a similar problem of 
Waring--Goldbach type in fewer variables. Thus, for example, since every representation 
of an even integer $n$ as the sum of three primes reduces to a representation of $n-2$ 
as the sum of two primes, we investigate the equation \eqref{1.1} with $k=1$ and 
$s=3$ only when $n$ is odd. We direct the reader to recent work \cite{But2010, Chu2009} 
for more on the Waring--Goldbach problem in the absence of such restrictions.\par

The first general bound for $H(k)$ was obtained by Hua \cite{Hua38}, who showed that
\begin{equation}\label{1.2}
H(k)\le 2^k+1\qquad (k \ge 1).
\end{equation}
This result, which generalizes I. M. Vinogradov's celebrated three primes theorem 
\cite{IVin37}, remains the best known bound on $H(k)$ for $k=1$, $2$ and $3$. When 
$k\ge 4$, on the other hand, the bound \eqref{1.2} has been sharpened considerably. 
These improvements may be grouped into three chronological phases: (i) work of 
Davenport and Hua from the 1940s and 1950s (see Hua \cite{Hua65}); (ii) refinements 
of the diminishing ranges method in Waring's problem developed in the mid-1980s by 
Thanigasalam \cite{Than85, Than87} and Vaughan \cite{Vaug86}; and (iii) more recent 
refinements of Zhao \cite{Zhao14}, the first author \cite{Kumc05}, and of Kawada and 
the second author \cite{KaWo01}. Thus, for intermediate values of $k$, the current state 
of play may be summarised with the bounds
\begin{gather*}
H(4) \le 13, \quad H(5) \le 21, \quad H(6) \le 32, \quad H(7) \le 46, \\ 
H(8) \le 63, \quad H(9) \le 83, \quad H(10) \le 103.
\end{gather*}
Here we note that although Thanigasalam \cite{Than87} claims only the bound 
$H(10)\le 107$, it is clear that his methods establish the stronger bound recorded above. 
For larger values of $k$, Hua \cite{Hua59, Hua65} adapted ideas from Vinogradov's 
work on Waring's problem to show that
\begin{equation}\label{1.3}
H(k)\le k(4\log k+2\log\log k+O(1)),\quad \text{as } k \to \infty.
\end{equation} 

In this paper, we make use of the new estimates for Weyl sums that result from recent 
work of the second author \cite{Wool12,Wool15} concerning Vinogradov's mean-value 
theorem to improve on the above results for $k \ge 8$. In particular, we obtain the 
following theorem, which represents the first improvement on Hua's bound \eqref{1.3} 
in more than half a century.

\begin{theorem}\label{th1}
When $k$ is large, one has $H(k)\le (4k-2)\log k+k-7$.
\end{theorem}

Our computations suggest strongly that the bound recorded in this theorem holds as soon 
as $k$ exceeds $64$. We also obtain the following bounds for $H(k)$ when $8 \le k \le 20$.

\begin{theorem}\label{th2}
Let $8\le k\le 20$. Then $H(k)\le s(k)$, where $s(k)$ is defined by Table~\ref{tab1}.
\begin{table}[h]
\begin{center}
\begin{tabular}{cccccccccccccc}
\toprule
$k$ & $8$  & $9$  & $10$ & $11$  & $12$  & $13$  & $14$  & $15$  & $16$  & $17$  & 
$18$  & $19$  & $20$ \\
$s(k)$ & $61$ & $75$ & $89$ & $103$ & $117$ & $131$ & $147$ & $163$ & $178$ & 
$194$ & $211$ & $227$ & $244$\\
\bottomrule
\end{tabular}\\[6pt]
\end{center}
\caption{Upper bounds for $H(k)$ when $8\le k\le 20$}\label{tab1}
\end{table}
\end{theorem}

We remark that we have an alternate approach to bounding $H(k)$ for larger $k$ that yields 
a bound in which, for all $k\ge 20$, at most three extra variables are required relative to the 
conclusion of Theorem \ref{th1}. In combination with Theorem \ref{th2} and the 
above-cited conclusions of earlier scholars, therefore, it follows that for every exponent $k$ 
with $k\ge 3$, one has
\[
H(k)\le (4k-2)\log k+k-4.
\]

Following some discussion of basic generating functions in \S2, we adapt Vaughan's 
variant of the diminishing ranges argument in \S3 so as to accommodate recent progress 
on Vinogradov's mean value theorem. In \S4 we apply these ideas to derive the mean 
value estimates underpinning the proofs of Theorems \ref{th1} and \ref{th2}. We 
complete the proof of the latter in \S5.\par

Throughout this paper, the letter $\eps$ denotes a sufficiently small positive number. 
Whenever $\eps$ occurs in a statement, we assert that the statement holds for each 
positive $\eps$, and any implied constant in such a statement is allowed to depend on 
$\eps$. The letter $p$, with or without subscripts, is reserved for prime numbers, and 
$c$ denotes an absolute constant, not necessarily the same in all occurrences. We also 
write $e(x)$ for $\exp(2\pi \mathrm{i}x)$, and $(a, b)$ for the greatest common divisor 
of $a$ and $b$. Finally, for real numbers $\theta$, we denote by $\lfloor \theta\rfloor$ the 
largest integer not exceeding $\theta$, and by $\lceil \theta\rceil$ the least integer no 
smaller than $\theta$.\par

We use several decompositions of the unit interval into sets of major and minor arcs. In 
order to facilitate discussion, we introduce some standard notation with which to describe 
such Hardy-Littlewood dissections. When $1\le Y\le X$, we define the set of major arcs 
$\mathfrak M(Y,X)$ as the union of the intervals
\[
\mathfrak M(q,a;Y,X)=\left\{ \alpha \in [0, 1): \; |q\alpha-a| \le X^{-1}\right\}
\]
with $0\le a\le q\le Y$ and $(a, q) = 1$. We define the corresponding set of minor arcs 
by putting $\mathfrak m(Y,X)=[0,1)\setminus \mathfrak M(Y,X)$.

\section{Bounds on exponential sums} Recent progress on Vinogradov's mean value 
theorem obtained by the second author \cite{Wool12, Wool13, Wool14, Wool15} permits 
improvements to be made in bounds on exponential sums of Weyl type. In this section, we 
collect together several such improved estimates for later use. Recall the classical Weyl sum
\[
f_k(\alpha; X)=\sum_{X < x \le 2X}e\left( \alpha x^k \right),
\]
in which we suppose that $k \ge 2$ is an integer and $\alpha$ is real. We define the 
exponent $\Sigma (k)$ as in Table \ref{tab2x} when $7\le k\le 20$, and otherwise by 
putting
\[
\Sigma(k)=2k^2-6k+4.
\]
Note that whenever $k\ge 3$, one has $\Sigma(k)\le 2k^2-6k+4$. Then, when $k \ge 3$ 
is an integer, we define $\sigma_k$ by means of the relation
\begin{equation}\label{2.1.0}
\sigma_k^{-1} = \min\big\{ 2^{k-1}, \Sigma(k)\big\}. 
\end{equation}

\begin{table}[h]
\begin{center}
\begin{tabular}{cccccccc}
\toprule
$k$  & $7$  & $8$  & $9$   & $10$   & $11$   & $12$   & $13$  \\
$\Sigma(k)$ & $58.093$ & $80.867$ & $107.396$ & $137.763$ & $172.027$ & 
$210.222$ & $252.370$\\
\midrule
$k$     & $14$  & $15$   & $16$   & $17$   & $18$   & $19$   & $20$ \\
$\Sigma(k)$ & $298.487$ & $348.580$ & $402.655$ & $460.718$ & $522.771$ & 
$588.815$ & $658.854$ \\
\bottomrule
\end{tabular}\\[6pt]
\end{center}
\caption{Definition of $\Sigma(k)$ for $7 \le k \le 20$}\label{tab2x}
\end{table}

For $k\ge 3$, we define the multiplicative function $w_k(q)$ by taking
\[
w_k(p^{uk+v})=\begin{cases} kp^{-u-1/2},&\text{when $u\ge 0$ and $v=1$,}\\
p^{-u-1},&\text{when $u\ge 0$ and $2\le v\le k$,}
\end{cases}
\]
and we note that $q^{-1/2}\le w_k(q)\ll q^{-1/k}$. We are now equipped to announce 
our main tool in the shape of the following lemma.   

\begin{lemma}\label{lem2.1.1}
Suppose that $k\ge 3$. Then either one has
\begin{equation}\label{2.1.1}
f_k(\alpha; X)\ll X^{1-\sigma_k+\eps},
\end{equation}
or there exist integers $a$ and $q$ such that
\[
1\le q\le X^{k\sigma_k},\quad (a, q)=1\quad \text{and}\quad 
|q\alpha-a|\le X^{-k+k\sigma_k},
\]
in which case 
\[
f_k(\alpha;X)\ll \frac {w_k(q)X}{1+X^k|\alpha -a/q|}+X^{1/2+\eps}.
\]
  
\end{lemma}

\begin{proof} Suppose first that $\alpha\in \mathfrak m(X,X^{k-1})$. Then the estimate 
(\ref{2.1.1}) follows at once from \cite[Theorem 11.5]{Wool14}, the refinement of 
\cite[Theorem 11.1]{Wool14} that follows by employing the bounds recorded in 
\cite[Theorem 1.2]{Wool15}, and Weyl's inequality (see \cite[Lemma 2.4]{Vaug97}). 
Meanwhile, when $\alpha\in \mathfrak M(X,X^{k-1})$, the desired conclusion follows by 
applying the argument of the proof of \cite[Lemma 2.1]{KaWo01}. The required details 
will be readily surmised from the special case $y=x$ of the proof of 
\cite[Lemma 2.2]{Kumc13}.
\end{proof}

We also require upper bounds for the corresponding Weyl sum over prime numbers,
\[
g_k(\alpha;X)=\sum_{X <p\le 2X}e\left( \alpha p^k \right),
\] 
and these we summarise in the next lemma.

\begin{lemma}\label{lem2.1.2}
Suppose that $k \ge 4$ and $X^{2\sigma_k/3} \le P \le X^{9/20}$. Then either one has
\begin{equation}\label{2.1.4}
g_k(\alpha; X) \ll X^{1-\sigma_k/3+\eps},
\end{equation}
or there exist integers $a$ and $q$ such that
\begin{equation}\label{2.1.5} 
1\le q\le P,\quad (a,q)=1\quad \text{and}\quad |q\alpha-a|\le PX^{-k},
\end{equation}
in which case 
\begin{equation}\label{2.1.6}
g_k(\alpha;X)\ll \frac{X^{1+\eps}}{(q+X^k|q\alpha-a|)^{1/2}}.
\end{equation}
\end{lemma}

\begin{proof}
First, when $\alpha\in \mathfrak m(P,P^{-1}X^k)$, the bound \eqref{2.1.4} follows from 
the special case $\theta=1$ of \cite[Theorem 1.2]{Kumc13}. Here, one replaces the 
exponent $\sigma_k$ of that paper with the refinement made available in (\ref{2.1.0}) 
by virtue of Lemma \ref{lem2.1.1}. We note in this context that the present absence of 
von Mangoldt weights is easily accommodated by applying the usual routine: one 
eliminates the prime powers $p^h$ $(h\ge 2)$, with an acceptable error term, and then 
applies partial summation to remove the remaining logarithmic weight. On the other 
hand, when instead $\alpha \in \mathfrak M(P,P^{-1}X^k)$, the hypotheses 
\eqref{2.1.5} are in play, and the inequality \eqref{2.1.6} is a direct consequence of 
\cite[Theorem~2]{Kumc06}.
\end{proof}

Finally, we have need of a variant of a lemma of Vaughan (see \cite[Lemma 1]{Vaug86}) 
dealing with an exponential sum over a difference polynomial, namely
\[
F_k(\alpha;X,H)=\sum_{1\le h\le H}\sum_{X<x\le 2X}e\left( \alpha \left( (x+h)^k-x^k 
\right) \right),
\]
where $1\le H\le X$.

\begin{lemma}\label{lem2.1.3}
Suppose that $k \ge 4$. Then either one has
\begin{equation}\label{2.1.7}
F_k(\alpha;X,H)\ll HX^{1-\sigma_{k-1}+\eps},
\end{equation}
or there exist integers $a$ and $q$ such that
\begin{equation}\label{2.1.8} 
1\le q\le X^{(k-2)\sigma_{k-1}},\quad (a, q) = 1\quad \text{and}\quad 
|q\alpha-a|\le X^{(k-2)\sigma_{k-1}}(HX^{k-1})^{-1},
\end{equation}
in which case
\begin{equation}\label{2.1.9}
F_k(\alpha;X,H)\ll \frac{q^{-1/(k-2)}HX^{1+\eps}}{1+HX^{k-1}|\alpha-a/q|}+
HX^{1/3+\eps}.
\end{equation}
\end{lemma}

\begin{proof}
Put $C=k^{3k}$ and let $Q=CHX^{k-2}$. Define $\mathfrak n$ to be the set of points 
$\alpha \in [0,1)$ with the property that whenever $a\in \mathbb Z$, $q\in \mathbb N$, 
$(a,q)=1$ and $|q\alpha-a|\le Q^{-1}$, then $q>X$. Suppose in the first instance that 
$k\ge 8$. We bound $|F_k(\alpha;X,H)|$ for $\alpha\in \mathfrak n$ by applying the 
method of proof of \cite[Lemma 10.3]{VW91}, in which we formally take 
$M=\tfrac{1}{2}$ and $R=2$. By substituting the conclusion of 
\cite[Theorem 11.5]{Wool14}, and the refinement of \cite[Theorem 11.1]{Wool14} 
utilising \cite[Theorem 1.2]{Wool15}, for \cite[Lemma 10.2]{VW91}, one obtains the 
bound
\begin{equation}\label{2.1.10}
\sup_{\alpha\in \mathfrak n}|F_k(\alpha;X,H)|\ll X^{1-\sigma_{k-1}+\eps}H.
\end{equation}
When $4\le k\le 7$, meanwhile, the bound (\ref{2.1.10}) with $\sigma_{k-1}=2^{2-k}$ 
follows from \cite[Lemma~1]{Vaug86}. This establishes (\ref{2.1.7}) when 
$\alpha\in \mathfrak n$.\par

Suppose next that $\alpha\not\in \mathfrak n$. By Dirichlet's theorem on Diophantine 
approximation, there exist integers $a$ and $q$ with
\[
1\le q\le Q,\quad (a,q)=1\quad \text{and}\quad |q\alpha-a|\le Q^{-1}.
\]
Since $\alpha \not\in \mathfrak n$, it follows that $q\le X$. Then 
\cite[Lemma 2]{Vaug86} yields the bound
\begin{equation}\label{2.1.13}
F_k(\alpha;X,H)\ll \frac {q^{-1/(k-2)}HX^{1+\eps}}{1+HX^{k-1}|\alpha -a/q|}+ 
Hq^{(k-2)/(k-1)+\eps}.
\end{equation}
If either inequality in \eqref{2.1.8} fails, then this implies \eqref{2.1.7} once more. 
Finally, when \eqref{2.1.8} holds, we have 
$q^{(k-2)/(k-1)}\le X^{(k-2)\sigma_{k-1}}\le X^{1/3}$, and \eqref{2.1.9} follows from 
\eqref{2.1.13}. This completes the proof of the lemma.
\end{proof}

\section{Mean-values for $k$th powers}
We describe in this section an enhanced diminishing ranges argument of Vaughan 
\cite{Vaug86} of use for mean values of intermediate and larger orders. Our first lemma is 
a variant of \cite[Theorem 3]{Vaug86} that makes use of Lemma \ref{lem2.1.3}.

\begin{lemma}\label{lem2.2.2}
Let $k\ge 4$ be an integer, and define $\sigma_k$ by means of the relation 
\eqref{2.1.0}. Also, let $s=\left\lfloor \frac 12(k+3)\right\rfloor$, consider real numbers 
$\lambda_1,\dots,\lambda_s$ with
\[
\lambda_1=1,\quad 1\ge \lambda_2\ge 1 - 1/k,\quad \lambda_2 \ge 
\lambda_i>1/2\quad (2 \le i \le s),
\]
and set $\nu=k\lambda_2-k+1$. Consider a large real number $N$, and put 
$N_i=N^{\lambda_i}$ $(1\le i\le s)$. Let $R(m)$ be a non-negative arithmetic function. 
Finally, with $C = C(k, \eps) \ge 2$, define
\[
\mathcal G(\alpha)=\sum_{1\le m\le CN_2^k}R(m)e(\alpha m)
\]
and
\[
\mathcal F_j(\alpha)=\mathcal G(\alpha)\prod_{i = j}^sf_k(\alpha;N_i)\quad (j = 1,2).
\]
Then
\begin{equation}\label{4.w}
\int_0^1|\mathcal F_1(\alpha)|^2\,{\rm d}\alpha \ll \left( 
N+N^{1+\nu-\sigma_{k-1}+\eps}\right)\int_0^1|\mathcal F_2(\alpha)|^2\,{\rm d}
\alpha +\mathcal F_1(0)^2N^{\eps-k}.
\end{equation}
\end{lemma}

\begin{proof}
Write
$$R_1(n)=\sum_{N_1<x_1\le 2N_1}\ldots \sum_{N_s<x_s\le 2N_s}
\sum_{1\le m\le (2N_2)^k}R(m),$$
in which the summation is subject to the condition $n=m+x_1^k+\ldots +x_s^k$. Then 
it follows via orthogonality that the mean value on the left hand side of (\ref{4.w}) is 
bounded above by $\sum_nR_1(n)^2$. We may therefore follow the argument of the 
proof of \cite[Theorem 3]{Vaug86} leading to formula (2.8) of the latter source. That 
formula defines the  integral
\[
M=\int_0^1 F_k(\alpha;N,H)|\mathcal F_2(\alpha)|^2\, {\rm d}\alpha,
\]
where $H=\min \{CN^{\nu},N\}$.\par

Define the set of major arcs $\mathfrak N$ to be the union of the intervals
\[
\mathfrak N(q,a)=\mathfrak M
(q,a;N^{(k-2)\sigma_{k-1}},HN^{k-1-(k-2)\sigma_{k-1}}),
\]
over integers $a$ and $q$ satisfying (\ref{2.1.8}). Also, define the function 
$G_1(\alpha)$ on $[0,1)$ by taking
\[
G_1(\alpha)=\frac{q^{-1/(k-2)}HN}{1+HN^{k-1}|\alpha-a/q|},
\]
when $\alpha\in \mathfrak N(q,a)\subseteq \mathfrak N$, and otherwise by putting 
$G_1(\alpha)=0$. Then by applying Lemma \ref{lem2.1.3}, we obtain the estimate
\begin{equation}\label{4.wc}
M\ll N^\eps \int_{\mathfrak N}G_1(\alpha)|\mathcal F_2(\alpha)|^2\,{\rm d}\alpha+
N^{1+\nu -\sigma_{k-1}+\eps}\int_0^1|\mathcal F_2(\alpha)|^2\, {\rm d}\alpha .
\end{equation}
This inequality replaces \cite[equations (2.17)--(2.21)]{Vaug86}. To complete the proof 
of the lemma, we follow the remainder of Vaughan's argument on 
\cite[pages 451--452]{Vaug86}, noting that the first integral on the right hand side of 
\eqref{4.wc} is the quantity $K$ defined in \cite[equation (2.21)]{Vaug86}. We remark 
that although our set of major arcs $\mathfrak N$ is somewhat larger than the 
respective set in Vaughan's paper, this does not pose any problems, since we 
nonetheless have $(k-2)\sigma_{k-1}\le \frac 12$, which serves as a satisfactory 
substitute for the relevant bound $(k-2)2^{2-k}\le \frac 12$ at the top of 
\cite[page 452]{Vaug86}).
\end{proof}

Before introducing our basic mean-value estimate, we define a set of admissible exponents 
for $k$th powers as follows. Let $t=t_k$ and $u=u_k$ be positive integral parameters to 
be chosen in due course. Then, with $\theta =1-1/k$, we first set 
\begin{equation}\label{lam.1}
\lambda_i=(\theta +\sigma_{k-1}/k)^{i-1}\quad (1\le i\le u+1).
\end{equation}
Finally, we define $\lambda_{u+2}, \dots, \lambda_{u+t}$ by putting
\begin{align}
\lambda_{u+2}&=\frac{k^2-\theta^{t-3}}{k^2+k-k\theta^{t-3}}\lambda_{u+1},
\label{lam.2}\\
\lambda_{u+j}&=\frac {k^2-k-1}{k^2+k-k\theta^{t-3}}\theta^{j-3}\lambda_{u+1}\quad 
(3\le j\le t),\label{lam.3}
\end{align}
and then set
\begin{equation}\label{lam.4}
\Lambda=\lambda_1+\ldots +\lambda_{t+u}.
\end{equation}

\begin{lemma}\label{lem2.3.1}
Let $k$, $t$ and $u$ be positive integers with $k\ge 3$ and 
$t\ge \left\lfloor \frac 12(k+3)\right\rfloor$, and let $v$ be a non-negative real number. 
Define the exponents $\lambda_j$ and $\Lambda$ by means of 
\eqref{lam.1}-\eqref{lam.4}. Then, when $N$ is large,
\begin{equation}\label{2.3.1}
\int_0^1 |f_k(\alpha; N)|^v\prod_{j=1}^{t+u}\left| g_k (\alpha;N^{\lambda_j})\right|^2 
\, {\rm d}\alpha \ll N^{2\Lambda+v-k+\eps}\left( 1+N^{k-\Lambda-v\sigma_k}\right).
\end{equation}
\end{lemma}

\begin{proof} 
We establish the bound \eqref{2.3.1} through an application of the Hardy-Littlewood 
method. We begin by examining an auxiliary mean value, establishing the bound
\begin{equation}\label{2.3.6a}
\int_0^1\prod_{j=i}^{t+u}|f_k(\alpha;N^{\lambda_j})|^2 \, {\rm d}\alpha \ll 
N^{2\Lambda_i-k+\eps}\left( 1+N^{k-\Lambda_i}\right),
\end{equation}
for $1\le i\le u+1$, where $\Lambda_i=\lambda_i+\ldots +\lambda_{t+u}$. Observe first 
that, by orthogonality, the bound
\[
\int_0^1\prod_{j=u+1}^{u+t}\left| f_k \left( \alpha;N^{\lambda_j}\right)\right|^2\, 
{\rm d}\alpha \ll N^{\Lambda_{u+1}+\eps}
\]
follows as a direct consequence of Vaughan~\cite[Theorem 6.1]{Vaug97}. When 
$1\le i \le u$, meanwhile, we apply Lemma \ref{lem2.2.2} with 
$\lambda_2=\theta+\sigma_{k-1}/k$ and $R(m)$ equal to the number of integral 
representations of the integer $m$ in the form
$$m=x_{s+1}^k+\ldots +x_{t+u}^k,$$
with $N^{\lambda_j}<x_j\le 2N^{\lambda_j}$ $(s+1\le j\le t+u)$. Since one then has 
$k\lambda_2-k+1=\sigma_{k-1}$, we deduce that whenever the estimate \eqref{2.3.6a} 
holds for $i=I+1$, then it holds also for $i=I$. The desired bound \eqref{2.3.6a} 
therefore follows for $1\le i\le u+1$ by backwards induction, starting from the base case 
$i=u+1$.\par

Now, with $P = N^{k\sigma_k}$, put $\mathfrak M = \mathfrak M(P, N^kP^{-1})$ and 
$\mathfrak m=[0,1)\setminus \mathfrak M$. Then by Lemma \ref{lem2.1.1},
\[
\sup_{\alpha \in \mathfrak m}|f_k(\alpha;N)|\ll N^{1-\sigma_k+\eps}.
\]
Furthermore, by comparing the underlying Diophantine equations, we have
\[
\int_0^1\prod_{j=1}^{t+u}|g_k(\alpha; N^{\lambda_j}) |^2 \, {\rm d}\alpha \le 
\int_0^1\prod_{j=1}^{t+u}|f_k(\alpha;N^{\lambda_j})|^2 \, {\rm d}\alpha.
\]
We therefore deduce from \eqref{2.3.6a} that
\begin{align}
\int_{\mathfrak m}|f_k(\alpha;N)|^v\prod_{j=1}^{t+u}\left|g_k(\alpha;N^{\lambda_j}) 
\right|^2\, {\rm d}\alpha &\ll \Bigl( \sup_{\alpha \in \mathfrak m}|f_k(\alpha;N)|
\Bigr)^v\int_{\mathfrak m}\prod_{j=1}^{t+u}\left|g_k(\alpha;N^{\lambda_j}) 
\right|^2\, {\rm d}\alpha\notag \\
&\ll N^{2\Lambda+v-k+\eps}(N^{-v\sigma_k}+N^{k-\Lambda-v\sigma_k}).
\label{2.3.4}
\end{align}

\par In order to estimate the contribution of the major arcs $\mathfrak M$ to the left side 
of \eqref{2.3.1}, we note that when $\alpha \in \mathfrak M(q, a; P, N^kP^{-1})\subseteq 
\mathfrak M$, it follows from Lemmata \ref{lem2.1.1} and \ref{lem2.1.2} that
\[
f_k(\alpha;N)\ll \frac{q^{-1/k}N}{1+N^k|\alpha - a/q|},
\]
and
\[
g_k(\alpha;N_j)\ll \frac{q^{-1/2}N_j^{1+\eps}}{(1+N_j^k|\alpha-a/q|)^{1/2}}\quad 
(j = 1,2).
\]
Applying these two inequalities in combination with a trivial bound for $g_k(\alpha; N_j)$ 
$(j\ge 3)$, we find that 
\[
\int_{\mathfrak M}|f_k(\alpha;N)|^v\prod_{j=1}^{t+u}\left|g_k( \alpha;N^{\lambda_j}) 
\right|^2\, {\rm d}\alpha \ll \sum_{\substack{0\le a\le q\le P\\ (a,q)=1}}
\int_{\mathfrak M(q,a)}\frac{q^{-2-v/k}N^{2\Lambda+v+\eps}}{(1+N^k
|\alpha -a/q|)^{1+v}}.
\]
Thus we arrive at the estimate
\[
\int_{\mathfrak M}|f_k(\alpha;N)|^v\prod_{j=1}^{t+u}\left|g_k( \alpha;N^{\lambda_j}) 
\right|^2\, {\rm d}\alpha \ll N^{2\Lambda+v-k+\eps},
\]
which, in combination with \eqref{2.3.4}, confirms the desired bound \eqref{2.3.1}.
\end{proof}

An immediate consequence of Lemma \ref{lem2.3.1} supplies useful mean value estimates 
in the Waring--Goldbach problem.

\begin{lemma}\label{lem2.3.2}
Let $k$, $t$ and $u$ be positive integers with $k\ge 3$ and 
$t\ge \left\lfloor \frac 12(k+3)\right\rfloor$, and let $w$ be a non-negative integer. Define 
the exponents $\lambda_j$ and $\Lambda$ by means of \eqref{lam.1}-\eqref{lam.4}, and 
put $\eta=\max\{0,k-\Lambda -2w\sigma_k\}$. Then when $N$ is sufficiently large, one 
has
\[
\int_0^1 |g_k(\alpha; N)|^{2w}\prod_{j=1}^{t+u}\left| 
g_k (\alpha;N^{\lambda_j})\right|^2 \, {\rm d}\alpha \ll 
N^{2\Lambda+2w-k+\eta+\eps}.
\]
\end{lemma}

\begin{proof} By considering the underlying Diophantine equation, it follows via 
orthogonality that the mean value in question is bounded above by
\[
\int_0^1 |f_k(\alpha; N)|^{2w}\prod_{j=1}^{t+u}\left| 
g_k (\alpha;N^{\lambda_j})\right|^2 \, {\rm d}\alpha .
\]
Then we deduce from Lemma \ref{lem2.3.1} 
that
\[
\int_0^1 |g_k(\alpha; N)|^{2w}\prod_{j=1}^{t+u}\left| 
g_k (\alpha;N^{\lambda_j})\right|^2 \, {\rm d}\alpha \ll 
N^{2\Lambda+2w-k+\eps}\left( 1+N^{k-\Lambda-2w\sigma_k}\right),
\]
and the proof of the lemma is complete.
\end{proof}

\section{An upper bound for $H(k)$}
An upper bound for $H(k)$ follows by combining the mean value estimate supplied by 
Lemma \ref{lem2.3.2} with the Weyl-type estimate stemming from Lemma 
\ref{lem2.1.2}. In certain circumstances an extra variable can be saved by employing the 
device of Zhao \cite[equation (3.10)]{Zhao14}. In this section we prepare a general lemma 
that captures both the results stemming from the basic strategy, and those reflecting the 
refinement stemming from the argument of Zhao.\par

\begin{lemma}\label{lemw.1}
Let $k$, $t$ and $u$ be positive integers with $k\ge 3$ and 
$t\ge \left\lfloor\frac 12(k+3)\right\rfloor$. Define the exponent $\Lambda$ by means of 
\eqref{lam.4}, and put
\[
v=\lfloor (k-\Lambda)/(2\sigma_k)\rfloor\quad \text{and}\quad \eta^*=
k-\Lambda -2v\sigma_k.
\]
Finally, define
\[
h=\begin{cases}1,&\text{when $0\le \eta^*<\tfrac{1}{2}\sigma_k$,}\\
2,&\text{when $\tfrac{1}{2}\sigma_k\le \eta^*<\sigma_k$,}\\
3,&\text{when $\sigma_k\le \eta^*<2\sigma_k$.}\end{cases}
\]
Suppose in addition that $2(t+u+v)+h\ge 3k+1$ and, when $h\in \{1,2\}$, that either 
$v\ge 3$ or $\eta^*<h\sigma_k/3$. Then
\[
H(k)\le 2(t+u+v)+h.
\]
\end{lemma}

\begin{proof} Let $s=2(t+u+v)+h$, and note that we are permitted to assume that 
$s\ge 3k+1$. Consider a large integer $n$ satisfying the congruence condition 
$n\equiv s\pmod{K(k)}$. Set $\lambda_j=1$ for $j>u+t$, and write 
$N={\textstyle \frac 12}n^{1/k}$. Denote by $R_{k,s}(n)$ the number of representations 
of $n$ in the form
\[
n=p_1^k+p_2^k+\dots +p_s^k,
\]
subject to $N^{\lambda_{j+1}}<p_{2j+\omega}\le 2N^{\lambda_{j+1}}$ for 
$0\le j\le \tfrac{1}{2}(s-\omega)$ and $\omega\in\{1,2\}$. Finally, write
\[
\mathcal G(\alpha)=\prod_{j=1}^{u+t+v}g_k(\alpha; N^{\lambda_j}).
\]
Then 
\begin{equation}\label{4.2w}
R_{k,s}(n)=\int_0^1 g_k(\alpha;N)^h\mathcal G(\alpha)^2e(-\alpha n)\, {\rm d}\alpha. 
\end{equation}

We now dissect the unit interval into sets of major and minor arcs. Let
\[
P=N^{1/3},\quad Q=N^kP^{-1},\quad L=\log N,\quad 
X=N^{2\Lambda+2v+h-k}L^{-s},
\]
in which $\Lambda$ is defined via (\ref{lam.4}). We choose the set of major arcs to be 
$\mathfrak M = \mathfrak M(P, Q)$ and write $\mathfrak m=[0,1)\setminus \mathfrak M$. 
The major arc contribution to the integral in \eqref{4.2w} can be approximated using the 
methods in Chapters 6 and 9 of a forthcoming monograph by Liu and Zhan \cite{LiZh14}. 
Alternatively, we may refer to the main theorem in Liu \cite{JLiu12}, which establishes that 
for any fixed $A>0$, one has the asymptotic formula
\begin{equation}\label{3.3}
\int_{\mathfrak M}g_k(\alpha;N)^h\mathcal G(\alpha)^2e(-\alpha n)\, {\rm d}\alpha 
=\mathfrak S_{k,s}(n)J_{k,s}(n)+O\left( XL^{-A}\right), 
\end{equation}
where
\[
\mathfrak S_{k,s}(n)=\sum_{q=1}^\infty \sum^q_{\substack{a=1\\ (a,q)=1}}
\phi(q)^{-s}\biggl(\sum^q_{\substack{r=1\\ (r,q)=1}}e(ar^k/q)\biggr)^se(-an/q)
\]
and
\[
J_{k,s}(n)=\int_{-\infty}^\infty V(\beta;N)^h\biggl( \prod_{i=1}^{u+t+v}
V(\beta;N^{\lambda_j})\biggr)^2e(-\beta n)\, {\rm d}\beta ,
\]
in which
\[
V(\beta;Z)=\int_Z^{2Z}\frac{e(\beta \gamma^k)}{\log \gamma}\, {\rm d}\gamma .
\]
Here, the expression $\mathfrak S_{k,s}(n)$ is the singular series associated with sums of 
$s$ $k$th powers of primes and $J_{k,s}(n)$ is the singular integral associated with the 
representations counted by $R_{k,s}(n)$ (see Liu and Zhan 
\cite[equations (4.50) and (9.16)]{LiZh14}). We remark that while the main result in Liu 
\cite{JLiu12} states \eqref{3.3} for a different set of major arcs, corresponding to a larger 
choice of $P$ than that given above, Liu's argument applies also to $\mathfrak M$. 
Moreover, with our choice of $N$, standard estimates for the singular series and the 
singular integral (see Liu and Zhan \cite[\S\S4.6 and 6.2]{LiZh14}) confirm that for all 
sufficiently large integers $n$ with $n\equiv s\pmod{K(k)}$,
\begin{equation}\label{3.4}
X\ll \mathfrak S_{k,s}(n)J_{k,s}(n)\ll X.
\end{equation}
We note in this context that the conditions $s\ge 3k+1$ and $n\equiv s\pmod{K(k)}$ 
ensure that the singular series is positive (see \cite[Theorem 12]{Hua65}).\par

We turn next to the contribution of the minor arcs, first considering the situation with 
$\sigma_k\le \eta^*<2\sigma_k$, in which case $h=3$, together with that in which 
$h\in \{1,2\}$ and $\eta^*<h\sigma_k/3$. By Lemma \ref{lem2.1.2}, one has
\begin{equation}\label{4.3}
\sup_{\alpha \in \mathfrak m}|g_k(\alpha;N)|\ll N^{1-\sigma_k/3+\eps}.
\end{equation}
Write
\begin{equation}\label{4.3w}
\Theta=\int_0^1|g_k(\alpha;N)\mathcal G(\alpha)|^2\,{\rm d}\alpha .
\end{equation}
Then, since the definition of $v$ ensures that $k-\Lambda<(2v+2)\sigma_k$, we 
find from Lemma \ref{lem2.3.2} that
\begin{equation}\label{4.3a}
\Theta \ll N^{2\Lambda +2v+2-k+\eps}.
\end{equation}
Thus, when $h=3$, we obtain the bound
\begin{align}\label{4.5}
\int_{\mathfrak m}|g_k(\alpha; N)^h\mathcal G(\alpha)^2|\, d\alpha \ll 
\Bigl(\sup_{\alpha \in \mathfrak m}|g_k(\alpha;N)|\Bigr)\Theta \ll XN^{-\sigma_k/4}.
\end{align}
On the other hand, when $h\in \{1,2\}$ and $\eta^*<h\sigma_k/3$, we find instead that
\[
\int_{\mathfrak m}|g_k(\alpha; N)^h\mathcal G(\alpha)^2|\, d\alpha \ll 
\Bigl(\sup_{\alpha \in \mathfrak m}|g_k(\alpha;N)|\Bigr)^h\int_0^1
|\mathcal G(\alpha)|^2\,{\rm d}\alpha \ll XN^{\eta^*-h\sigma_k/3+\eps}.
\]
By combining these estimates with (\ref{4.2w})-(\ref{3.4}), we conclude in these cases 
that
\begin{equation}\label{asy}
R_{k,s}(n)=\mathfrak S_{k,s}(n)J_{k,s}(n)+O(XL^{-A})\gg X.
\end{equation}
Thus $H(k)\le s$, completing the proof of the lemma when 
$\sigma_k\le \eta^*<2\sigma_k$.\par

The final case to consider is that in which $h\in \{1,2\}$ and $0\le \eta^*<\sigma_k$. 
Here, we employ the method of Zhao \cite[equation (3.10)]{Zhao14}. For simplicity of 
exposition, we provide a detailed account of the situation with $h=2$, although it will be 
clear how to adjust the argument to handle the case $h=1$. We begin with the observation 
that
\[
\int_{\mathfrak m}g_k(\alpha;N)^h\mathcal G(\alpha)^2e(-n\alpha)\,{\rm d}\alpha =
\sum_{N<p_1,p_2\le 2N}\int_{\mathfrak m}\mathcal G(\alpha)^2
e((p_1^k+p_2^k-n)\alpha)\,{\rm d}\alpha ,
\]
whence, by Cauchy's inequality,
\begin{equation}\label{ups.1}
\biggl| \int_{\mathfrak m}g_k(\alpha;N)^h\mathcal G(\alpha)^2e(-n\alpha)\,{\rm d}
\alpha \biggr|\le N^{h/2}\Upsilon^{1/2},
\end{equation}
where
\begin{align*}
\Upsilon&=\sum_{N<x_1,x_2\le 2N}\biggl| \int_{\mathfrak m}\mathcal G(\alpha)^2
e((x_1^k+x_2^k-n)\alpha)\,{\rm d}\alpha \biggr|^2\\
&=\int_{\mathfrak m}\int_{\mathfrak m}\mathcal G(\alpha)^2\mathcal G(-\beta)^2
f_k(\alpha-\beta;N)^he(-n(\alpha-\beta))\,{\rm d}\alpha \, {\rm d}\beta .
\end{align*}

\par Put $\mathfrak N(q,a)=\mathfrak M(q,a;N^{k\sigma_k},N^{k-k\sigma_k})$ and 
$\mathfrak N=\mathfrak M(N^{k\sigma_k},N^{k-k\sigma_k})$. Also, write 
$\mathfrak n=[0,1)\setminus \mathfrak N$. Next, denote by $\mathfrak B$ the set of 
ordered pairs $(\alpha,\beta)\in \mathfrak m^2$ for which 
$\alpha-\beta\in \mathfrak N\pmod{1}$, and put $\mathfrak b=\mathfrak m^2\setminus 
\mathfrak B$. Let $\Psi:[0,1)\rightarrow [0,\infty)$ denote the function defined by taking 
\[
\Psi(\alpha)=w_k(q)N(1+N^k|\alpha -a/q|)^{-1},
\]
when $\alpha \in \mathfrak N(q,a)\subseteq \mathfrak N$, and otherwise by taking 
$\Psi(\alpha)=0$. Then it follows from an application of the triangle inequality that
\[
\Upsilon\le \iint_{\mathfrak b}|f_k(\alpha-\beta;N)^h\mathcal G(\alpha)^2
\mathcal G(\beta)^2|\,{\rm d}\alpha \, {\rm d}\beta +\iint_{\mathfrak B}
|f_k(\alpha-\beta;N)^h\mathcal G(\alpha)^2\mathcal G(\beta)^2|\,{\rm d}\alpha 
\, {\rm d}\beta.
\]
Thus we deduce from Lemma \ref{lem2.1.1} that
\begin{equation}\label{4.z2}
\Upsilon\ll \Upsilon_1+N^{h-1}\Upsilon_2,
\end{equation}
where
\[
\Upsilon_1=N^{h-h\sigma_k+\eps}\int_0^1\int_0^1|\mathcal G(\alpha)
\mathcal G(\beta)|^2\,{\rm d}\alpha \, {\rm d}\beta .
\]
and
\[
\Upsilon_2=\iint_{\mathfrak B}\Psi(\alpha-\beta)|\mathcal G(\alpha)\mathcal G(\beta)|^2
\,{\rm d}\alpha \, {\rm d}\beta .
\]

It follows from Lemma \ref{lem2.3.2} that
\begin{equation}\label{4.z}
\Upsilon_1\ll N^{h-h\sigma_k+\eps}\left( N^{2\Lambda+2v-k+\eta^*+\eps}\right)^2
=X^2N^{-h+2\eta^*-h\sigma_k+3\eps}.
\end{equation}
In order to bound $\Upsilon_2$, we begin by using the estimate
\[
|g_k(\alpha ;N)g_k(\beta;N)|^2\ll |g_k(\alpha ;N)|^4+|g_k(\beta;N)|^4,
\]
in combination with trivial estimates and symmetry, to obtain
\[
\Upsilon_2\ll N^{\Lambda-1-\lambda_2}\int_{\mathfrak m}
\int_{\mathfrak m}\Psi(\alpha-\beta)| 
g_k(\alpha;N)^{v-1}g_k(\alpha;N^{\lambda_2})g_k(\beta;N)^2
\mathcal G(\alpha)\mathcal G(\beta)^2|\,{\rm d}\alpha \, {\rm d}\beta .
\]
Write
\[
\Phi=\sup_{\beta\in [0,1)}\int_0^1\Psi(\alpha-\beta)^2|g_k(\alpha;N^{\lambda_2})|^2
\,{\rm d}\alpha .
\]
Then since \cite[Lemma 2.2]{Zhao14} supplies the bound 
$\Phi \ll N^{2+2\lambda_2-k+\eps}$, it follows from Cauchy's inequality in combination 
with (\ref{4.3})-(\ref{4.3a}) that
\begin{align*}
\int_{\mathfrak m}\Psi(\alpha-\beta)| 
g_k(\alpha;N)^{v-1}g_k(\alpha;N^{\lambda_2})\mathcal G(\alpha)|\,{\rm d}\alpha 
&\ll \Bigl( \sup_{\alpha \in \mathfrak m}|g_k(\alpha ;N)|\Bigr)^{v-2}\Theta^{1/2} 
\Phi^{1/2}\\
&\ll (N^{1-\sigma_k/3})^{v-2}N^{\Lambda+\lambda_2+v+2-k+\eps}.
\end{align*}
We therefore conclude by means of (\ref{4.3a}) that
\begin{align*} 
\Upsilon_2&\ll N^{2\Lambda +2v-k-1-\sigma_k/4}\int_0^1|g_k(\beta;N)
\mathcal G(\beta)|^2\, {\rm d}\beta \\
&\ll N^{4\Lambda +4v+1-2k-\sigma_k/5}.
\end{align*}
On substituting this estimate together with (\ref{4.z}) into (\ref{4.z2}), and noting that, 
by hypothesis, we have $h\sigma_k>2\eta^*$, we deduce that for some positive number 
$\nu$, one has
\[
\Upsilon \ll X^2N^{-h-2\nu}.
\]
Inserting this bound into (\ref{ups.1}), we arrive at the estimate
\[
\int_{\mathfrak m}g_k(\alpha;N)^h\mathcal G(\alpha)^2e(-n\alpha)\,{\rm d}
\alpha \ll XN^{-\nu},
\]
an estimate that may be employed as a viable substitute for (\ref{4.5}) in the argument 
leading to (\ref{asy}). Thus the conclusion of the lemma follows also in these final cases, 
and so the proof of the lemma is complete.
\end{proof}

\section{Proof of Theorems \ref{th1} and \ref{th2}}
Theorems \ref{th1} and \ref{th2} are direct consequences of Lemma \ref{lemw.1}. 
Recall (\ref{lam.1})-(\ref{lam.4}), and write $\sigma=\sigma_{k-1}$ and 
$\phi=\theta+\sigma/k$. Then one has
\[
\sum_{i=1}^{u+1}\lambda_i=\frac{1-\phi^{u+1}}{1-\phi}=
\frac{k}{1-\sigma}(1-\phi^{u+1})
\]
and
\begin{align*}
\sum_{j=2}^t\lambda_{u+j}&=\left(\frac{k^2-\theta^{t-3}}
{k^2+k-k\theta^{t-3}}\right)\lambda_{u+1}+\left(\frac {k^2-k-1}
{k^2+k-k\theta^{t-3}}\right)\left( \frac{1-\theta^{t-2}}{1-\theta}\right)\lambda_{u+1}
\\
&=\left( \frac{k^3-k-(k^3-2k^2+2)\theta^{t-3}}{k^2+k-k\theta^{t-3}}\right) \phi^u.
\end{align*}
Thus
\[
\Lambda=\frac{k}{1-\sigma}+\left( \frac{(k^3-k-(k^3-2k^2+2)\theta^{t-3})
(1-\sigma)-(k-1+\sigma)(k^2+k-k\theta^{t-3})}{(k^2+k-k\theta^{t-3})(1-\sigma)}\right) 
\phi^u,
\]
and hence
\[
k-\Lambda=-\frac{k\sigma}{1-\sigma}+\left( \frac{k^2(k+1)\sigma +\theta^{t-3}
((k^3-3k^2+k+2)-\sigma(k^3-2k^2+k+2))}{(k^2+k-k\theta^{t-3})(1-\sigma)}\right) 
\phi^u.
\]
This formula provides the key input into our application of Lemma \ref{lemw.1}.

\begin{proof}[The proof of Theorem \ref{th2}] Let $k$ be an integer with $8\le k\le 20$, 
and define $t =t_k$, $u=u_k$, $v=v_k$ and $h=h_k$ according to Table \ref{tab2}. Then 
a straightforward computer program confirms that the hypotheses of Lemma \ref{lemw.1} 
hold, and hence that $H(k)\le 2(t+u+v)+h$. Indeed, with $h_k^*$ defined by 
Table \ref{tab3}, one finds for these values of $k$ that $2\eta^*/\sigma_k<h_k^*$. All 
entries in this table have been rounded up in the final decimal place recorded. This 
completes the proof of Theorem \ref{th2}.

\begin{table}[ht]
\begin{center}
\begin{tabular}{cccccccccccccc}
\toprule
$k$ & $8$ & $9$ & $10$ & $11$ & $12$ & $13$ & $14$ & $15$ & $16$ & $17$ & $18$ & 
$19$ & $20$ \\
\midrule
$t_k$ & $9$ & $18$ & $13$ & $14$ & $9$ & $20$ & $13$ & $14$ & $26$ & $28$ & $33$ 
& $23$ & $35$ \\
$u_k$ & $18$ & $15$ & $27$ & $30$ & $44$ & $41$ & $56$ & $64$ & $56$ & $62$ & 
$68$ & $81$ & $78$ \\
$v_k$ & $3$ & $4$ & $3$ & $7$ & $5$ & $4$ & $4$ & $3$ & $6$ & $6$ & $4$ & $9$ & 
$8$ \\
$h_k$ & $1$ & $1$ & $3$ & $1$ & $1$ & $1$ & $1$ & $1$ & $2$ & $2$ & $1$ & $1$ & 
$2$ \\
\bottomrule
\end{tabular}\\[12pt]
\end{center}
\caption{The values of $t_k$, $u_k$, $v_k$ and $h_k$ for $8 \le k \le 20$}\label{tab2}
\end{table}

\begin{table}[ht]
\begin{center}
\begin{tabular}{cccccccccccccc}
\toprule
$k$ & $8$ & $9$ & $10$ & $11$ & $12$ & $13$ & $14$ \\
\midrule
$h_k^*$ & $0.56062$ & $0.09534$ & $2.05276$ & $0.01726$ & $0.00008$ & 
$0.99878$ & $0.01987$\\
\bottomrule \toprule
$k$ & $15$ & $16$ & $17$ & $18$ & $19$ & $20$\\
\midrule
$h_k^*$ & $0.00055$ & $1.90169$ & $1.99481$ & $0.00497$ & $0.00294$ & 
$1.10563$\\
\bottomrule
\end{tabular}\\[12pt]
\end{center}
\caption{The values of $h_k^*$ for $8 \le k \le 20$}\label{tab3}
\end{table}

\end{proof}

It is evident that there is substantial non-monotonicity in the values of $t_k$, $u_k$ and 
$v_k$ recorded in Table \ref{tab2}. It seems to the authors that since $\theta$ and $\phi$ 
have values that are rather close together, then there is relatively little sensitivity in the 
optimisation to the specific values of $t_k$ and $u_k$, but rather it is the sum $t_k+u_k$ 
that is important. We note also that the values of $h_k^*$ are extremely small for a 
number of the exponents $k$, so that relatively modest improvements to the values 
$\Sigma (k)$ recorded in Table \ref{tab2x} will lead to improved values of $H(k)$.

\begin{proof}[The proof of Theorem \ref{th1}]
We may now suppose that $k$ is large. We put $t=t_k$ and $u=u_k$, where
\[
t_k=\left\lceil \tfrac{1}{2}k\log k\right\rceil \quad \text{and}\quad u_k=
\left\lceil 2k\log k\right\rceil -t-4.
\]
It is convenient for later use to put $\gam=\lceil 2k\log k\rceil -2k\log k$. Also, we write
\[
\tau=\frac{1}{2k^2-6k+4}\quad \text{and}\quad \sigma=\frac{1}{2k^2-10k+12},
\]
so that $\sigma_k=\tau$ and $\sigma_{k-1}=\sigma$. Our formula for $k-\Lambda$ may 
now be written in the shape
\[
k-\Lambda=-\frac{k\sigma}{1-\sigma}+\left( \frac{k^2(k+1)(k-1)^3\sigma +\theta^t
k^3(k^3-3k^2+O(k))}{(k-1)^3(k^2+k-k\theta^{t-3})(1-\sigma)}\right) 
\phi^u.
\]

\par Since
\[
\log \theta=\log \left( 1-\frac{1}{k}\right)=-\frac{1}{k}-\frac{1}{2k^2}+O\left( 
\frac{1}{k^3}\right),
\]
it follows that
\[
t\log \theta =-\frac{t}{k}-\frac{\log k}{4k}+O\left( \frac{\log k}{k^2}\right),
\]
and hence
\[
\theta^t=e^{-t/k}\left( 1-\frac{\log k}{4k}+O(k^{-3/2})\right)\asymp k^{-1/2}.
\]
Similarly, since
\[
\log \phi=\log \left( 1-\frac{1-\sigma}{k}\right)=-\frac{1}{k}-\frac{1}{2k^2}+
O\left(\frac{1}{k^3}\right),
\]
we have
\[
\phi^u=e^{-u/k}\left( 1-\frac{3\log k}{4k}+O(k^{-3/2})\right) \ll k^{-3/2}.
\]
In particular, we find that
\[
k-\Lambda=-\frac{k\sigma}{1-\sigma}+\left(k-1+O(k^{-1/2})\right)\theta^t\phi^u
+O(k^{-5/2}),
\]
and that
\begin{align*}
\theta^t\phi^u&=e^{-(t+u)/k}\left( 1-\frac{\log k}{k}+O(k^{-3/2})\right)\\
&=e^{(4-\gamma)/k}\left( \frac{1}{k^2}-\frac{\log k}{k^3}+O(k^{-7/2})\right).
\end{align*}

\par On noting that
\[
\frac{\sigma}{\tau}=\frac{2k^2-6k+4}{2k^2-10k+12}=1+\frac{2}{k}+O\left( 
\frac{1}{k^2}\right),
\]
it follows that
\begin{align*}
\frac{k-\Lambda}{2\tau}&=-\tfrac{1}{2}(k+2)+
\frac{e^{(4-\gam)/k}(k^2-3k+2)}{k^3}(k-1)(k-\log k)+O(k^{-1/2})\\
&=-\tfrac{1}{2}(k+2)+(k-\log k-4)\left( 1+\frac{4-\gam}{k}\right) +O(k^{-1/2})\\
&=\tfrac{1}{2}k-\log k-1-\gamma+O(k^{-1/2}).
\end{align*}
Let $v=\left\lfloor(k-\Lambda)/(2\tau)\right\rfloor$, put $\eta^*=k-\Lambda-2v\tau$, and 
define $h$ as in the statement of Lemma \ref{lemw.1}. In particular, one has 
$0\le \eta^*<2\tau$, and no matter what the value of $\eta^*$ may be, one confirms that
\[
2v+h=\frac{k-\Lambda-\eta^*}{\tau}+h\le \frac{k-\Lambda}{\tau}+2
\le k-2\log k-2\gamma+O(k^{-1/2}).
\]
Therefore, since
\[
2(t+u+v)+h\le 2(2k\log k+\gamma-4)+k-2\log k-2\gamma+O(k^{-1/2}),
\]
we conclude from Lemma \ref{lemw.1} that
\[
H(k)\le (4k-2)\log k+k-8+O(k^{-1/2}).
\]
In view of our assumption that $k$ is sufficiently large, it follows that 
\[
H(k)\le (4k-2)\log k+k-7,
\]
and the proof of Theorem \ref{th1} is complete.
\end{proof}

\providecommand{\bysame}{\leavevmode\hbox to3em{\hrulefill}\thinspace}

\end{document}